\documentclass[11pt, a4paper]{amsart}
\usepackage{amssymb, array, amsmath, amscd, pdfpages, enumerate, amsthm, fixltx2e, setspace, hyperref, bbm,overpic}
\usepackage[justification=centering,margin=0.5cm]{caption}

\DeclareMathOperator*{\Ker}{Ker}

\DeclareMathOperator*{\dist}{dist}

\DeclareMathOperator*{\R}{Re}
\DeclareMathOperator*{\I}{Im}

\newcommand{\dd}{\mathrm{d}}

\newcommand{\RR}{\mathbb{R}}
\newcommand{\CC}{\mathbb{C}}

\newcommand{\NN}{\mathbb{N}}

\newcommand{\T}{(T(t))_{t\ge 0}}

\newtheorem{thm}{Theorem}[section]

\theoremstyle{definition}
\newtheorem{rem}[thm]{Remark}

\numberwithin{equation}{section}

\begin{document}
\title[Optimal energy decay for the wave-heat system]{Optimal energy decay for the wave-heat system on a rectangular domain}
\author[C.J.K. Batty]{Charles Batty}
\address[C.J.K. Batty]{St John's College, St Giles, Oxford\;\;OX1 3JP, UK}
\email{charles.batty@sjc.ox.ac.uk}

\author[L. Paunonen]{Lassi Paunonen}
\address[L. Paunonen]{Department of Mathematics, Tampere University of Technology, PO.\ Box 553, 33101 Tampere, Finland}
\email{lassi.paunonen@tut.fi}

\author[D. Seifert]{David Seifert}
\address[D. Seifert]{St John's College, St Giles, Oxford\;\;OX1 3JP, UK}
\email{david.seifert@sjc.ox.ac.uk}

\begin{abstract}
We study the rate of energy decay for solutions of a coupled wave-heat system on a rectangular domain. Using techniques from the theory of $C_0$-semigroups, and in particular a well-known result due to Borichev and Tomilov, we prove that the energy of classical solutions decays like $t^{-2/3}$ as $t\to\infty$. This rate is moreover shown to be sharp. Our result implies in particular that a general estimate in the literature, which predicts at least logarithmic decay and is known to be best possible in general, is suboptimal  in the special case under consideration here. Our strategy of proof involves direct estimates based on separation of variables and a refined version of the technique developed in our earlier paper for a one-dimensional wave-heat system.

\end{abstract}

\thanks{This work was carried out while L.\ Paunonen visited Oxford from January to June 2017 and while D.\ Seifert  visited Tampere University of Technology in September 2017. The visits were funded by the Academy of Finland grants 298182 and 310489, respectively, held by L.\ Paunonen. All three authors would like to express their gratitude to the two anonymous referees whose thoughtful comments improved the paper.}

\maketitle

\section{Introduction}\label{sec:intro} 

In this paper we study the rate of energy decay for solutions of a coupled wave-heat system on a rectangular domain. Let $\Omega_-=(-1,0)\times(0,1)$ and $\Omega_+=(0,1)\times(0,1)$. We consider the following model:
\begin{equation}\label{eq:PDE}
\left\{\begin{aligned}
u_{tt}(x,y,t)&=\Delta u(x,y,t), \hspace{40pt}&(x,y)\in \Omega_-,\;t>0,\\
w_{t}(x,y,t)&=\Delta w(x,y,t), &(x,y)\in \Omega_+,\;t>0,\\[3pt]
u(-1,y,t)&=w(1,y,t)=0, &y\in(0,1),\;t>0,\\
u(x,0,t)&=u(x,1,t)=0, &x\in(-1,0),\; t>0,\\
w(x,0,t)&=w(x,1,t)=0, & x\in (0,1),\; t>0,\\
u_t(0,y,t)&=w(0,y,t),& y\in(0,1),\;t>0,\\
u_x(0,y,t)&=w_x(0,y,t),& y\in(0,1),\;t>0,\\[3pt]
u(x,y,0)&=u_0(x,y),&(x,y)\in\Omega_-,\\
u_t(x,y,0)&=v_0(x,y),&(x,y)\in\Omega_-,\\
w(x,y,0)&=w_0(x,y),&(x,y)\in\Omega_+,
\end{aligned}\right.
\end{equation}
for suitable initial data $u_0,v_0$ defined on $\Omega_-$ and $w_0$ defined on $\Omega_+$. Given an initial vector $z_0=(u_0,v_0,w_0)$ of initial data, the energy $E_{z_0}(t)$ at time $t\ge0$ of the corresponding solution is defined as 
$$E_{z_0}(t)=\frac12\int_\Omega\Big( |\nabla u(x,y,t)|^2+|u_t(x,y,t)|^2+|w(x,y,t)|^2\Big)\,\dd(x,y),\quad t\ge 0,$$
where $u$ and $w$ have been extended by zero to $\Omega=(-1,1)\times(0,1).$ For sufficiently regular solutions we see by a direct calculation that
$$ \dot{E}_{z_0}(t)=-\int_{\Omega_+}|\nabla w(x,y,t)|^2\,\dd(x,y),\quad t\ge0,$$
so the energy is non-increasing in time. The aim in this paper is to establish a sharp estimate for the \emph{rate} at which the energy of sufficiently regular solutions decays to zero as $t\to\infty$.

Questions of this type, often for more general domains, have received a considerable amount of attention over the last decade or so, going back at least to \cite{RauZha05}. The underlying motivation is to understand the physical phenomenon of \emph{structure-fluid} interaction, for which the wave-heat model serves as a linear, and hence more tractable, approximate model retaining many of the key features of more realistic models. Some of the most general results in this area are obtained in \cite{Duy07,ZhaZua07}, and they show in particular that the rate of energy decay one can expect depends crucially on the geometry of the wave and heat domains. In fact, as is customary in the theory of damped wave equations, one may consider billiard ball trajectories beginning at points in the wave domain and reflecting according to the laws of optics along the Dirichlet part of the boundary. According to \cite{Duy07,ZhaZua07}, if there exist so-called \emph{trapped  rays}, that is to say trajectories which never enter the heat domain, then one obtains a logarithmic rate of energy decay for classical solutions, and for general geometries this result is sharp. On the other hand, if every trajectory eventually enters the heat domain then for classical solutions one obtains $E_{z_0}(t)=o(t^{-(2-\varepsilon)})$ as $t\to\infty$ for every $\varepsilon>0$. In fact, it was shown recently in \cite{AvaLas16} using a delicate microlocal argument that if the heat domain completely surrounds the wave domain one may even take $\varepsilon=0$, which gives the best possible estimate; see also \cite{AvaTri13}. Our main objective in the present paper is to obtain a sharp estimate for the rate of energy decay of classical solutions to our wave-heat system~\eqref{eq:PDE}. Our result in particular illustrates that the theoretical logarithmic estimate for the rate of energy decay is in general not optimal in the case where, as in \eqref{eq:PDE}, we have a rectangular domain with  trapped rays. Our main results concern the rate of energy decay of so-called classical solutions of \eqref{eq:PDE}, to be introduced formally in Section~\ref{sec:sg} below, and may be summarised as follows.

\begin{thm}\label{thm:int}
For classical solutions of the system~\eqref{eq:PDE} with initial data vector $z_0$ the energy satisfies $E_{z_0}(t)=o(t^{-2/3})$ as $t\to\infty$, and moreover this rate is sharp.
\end{thm}

Our approach is based on the theory of $C_0$-semigroups, and in particular on the a result due to Borichev and Tomilov \cite{BT10}, which reduces the problem of determining the rate of energy decay to estimating the norm of the resolvent operator along the imaginary axis; see also \cite{AvaLas16, AvaTri13, BaPaSe16}. Our argument divides naturally into three steps. First, in Section~\ref{sec:sg}, we show that the  system in~\eqref{eq:PDE} is well posed in the sense of $C_0$-semigroups and we describe the spectrum $\sigma(A)$ of the corresponding infinitesimal generator $A$, showing in particular that $\sigma(A)$ is contained in the open left-half plane. Then in Section~\ref{sec:res} we turn to estimating the resolvent of $A$ along the imaginary axis. We first obtain an upper bound on the growth of the resolvent in Theorem~\ref{thm:ub} and then we prove in Theorem~\ref{thm:opt} that our upper bound is sharp. Finally, in Section~\ref{sec:energy} we put together the pieces and apply the Borichev-Tomilov theorem to obtain a suitable form of Theorem~\ref{thm:int}. Throughout Sections~\ref{sec:sg} and \ref{sec:res} we take advantage of the special rectangular geometry of our domain, which allows us to obtain optimal bounds by means of direct estimates as opposed, for instance, to indirect microlocal arguments. More specifically, we use separation of variables, thus decomposing the two-dimensional problem into a family of one-dimensional problems which can be dealt with using techniques akin to those developed in our earlier paper \cite{BaPaSe16}; see also \cite{BurHit07,Sta17} for similar arguments in the context of damped wave equations.  

Our notation is standard throughout. Given a closed operator $A$ on a Hilbert space, which will always be assumed to be complex, we denote its domain by $D(A)$ and its kernel by $\Ker(A)$. The spectrum of $A$ is denoted by $\sigma(A)$, and given $\lambda\in\CC\setminus\sigma(A)$ we write $R(\lambda,A)$ for the resolvent operator $(\lambda-A)^{-1}$. 
For real-valued quantities $p$ and $q$, we use the notation $p\lesssim q$ to indicate that $p\le Cq$ for some constant $C>0$ which is independent of all the parameters that are free to vary in a given situation. We write $p\asymp q$ if $p\lesssim q$ and $q\lesssim p$. Furthermore, we make use where convenient of standard asymptotic notation such as `big O', `little o' and $\sim$.
We let $\NN=\{1,2,3,\dots\}$ and write $\CC_-$ for the open left half-plane $\{\lambda\in\CC:\R\lambda<0\}$.

\section{Well-posedness}\label{sec:sg} 

Our approach to studying \eqref{eq:PDE} is based on the theory of $C_0$-semigroups, so we first recast our system  in the form of an abstract Cauchy problem. Let $\Gamma=\partial\Omega_-\cap\partial\Omega_+$, which is simply the vertical line segment $\{(0,y):0\le y\le 1\}$, and let $H_\Gamma^1(\Omega_\mp)$ denote the set of functions in $H^1(\Omega_\mp)$ whose trace vanishes on $\partial\Omega_\mp\setminus\Gamma$, endowed with the equivalent Poincar\'e norm $\|u\|=\|\nabla u\|_{L^2}$, $u\in H_\Gamma^1(\Omega_\mp)$. We let $Z$ denote the  Hilbert space $H_\Gamma^1(\Omega_-)\times L^2(\Omega_-)\times L^2(\Omega_+)$, endowed with its natural inner product, and let $Z_0$ denote the subspace of $Z$ consisting of all $(u,v,w)\in Z$ such that $\Delta u\in L^2(\Omega_-),\ v\in H^1_\Gamma(\Omega_-),\ w\in H^1_\Gamma(\Omega_+)$ and $\Delta w\in L^2(\Omega_+)$. Define the operator $A$ by $A(u,v,w)=(v, \Delta u,\Delta w)$ for $(u,v,w)$ in the domain
$$D(A)=\big\{(u,v,w)\in Z_0: v|_\Gamma=w|_\Gamma\mbox{ and }u_x|_\Gamma=w_x|_\Gamma\big\}$$
of $A$, where the coupling conditions along $\Gamma$ appearing in the definition of $D(A)$ are to be understood in the sense of traces. Then $A$ is a closed and densely defined operator on $Z$ but does not have compact resolvent; see \cite[Theorem~2]{ZhaZua07}. Letting $z(t)=(u(\cdot,t),u_t(\cdot,t),w(\cdot,t))$, $t\ge0$, we may rewrite the wave-heat system \eqref{eq:PDE} as
\begin{equation}\label{eq:ACP}
\left\{\begin{aligned}
\dot{z}(t)&=Az(t),\quad t\ge0,\\
z(0)&=z_0,
\end{aligned}\right.
\end{equation}
where $z_0\in Z$. Our first result, Theorem~\ref{thm:sg} below, establishes among other things that the operator $A$ is the infinitesimal generator of a $C_0$-semigroup $\T$ of contractions on $Z$. It follows that the unique solution of \eqref{eq:ACP} is given by $z(t)=T(t)z_0$, $t\ge0$. This solution in general satisfies \eqref{eq:ACP} only in the so-called mild sense, but it is a solution in the classical sense if, and in fact and only if, $z_0\in D(A)$; see for instance \cite{ABHN11} for details on the theory of $C_0$-semigroups. It is for such classical solutions that we shall establish, in Section~\ref{sec:energy} below, a sharp estimate on the rate of energy decay in the sense described in Section~\ref{sec:intro}. In what follows we choose the square root function with a branch cut along the negative real axis.

\begin{thm}\label{thm:sg}
The operator $A$ generates a $C_0$-semigroup $\T$ of contractions on $Z$ and $\sigma(A)\subseteq\CC_-$. Moreover, if we let $p_k(\lambda)=(k^2\pi^2+\lambda^2)^{1/2}$ and $q_k(\lambda)=(k^2\pi^2+\lambda)^{1/2}$ for $k\in\NN$ and $\lambda\in\CC$ then the point spectrum of $A$ satisfies $\sigma_p(A)=\bigcup_{k\in\NN}\Sigma_k$, where, for $k\in\NN$,
$$\Sigma_k=\left\{\lambda\in\CC_-:p_k(\lambda),q_k(\lambda)\ne0\mbox{ and }\lambda\frac{\tanh p_k(\lambda)}{p_k(\lambda)}+\frac{\tanh q_k(\lambda)}{q_k(\lambda)}=0\right\}.$$

\end{thm}

\begin{proof}
For $z_0=(u,v,w)\in D(A)$ a simple calculation using Green's theorem on each of the regions $\Omega_-$ and $\Omega_+$ shows that
\begin{equation}\label{eq:diss}
\R\langle Az_0,z_0\rangle=-\int_{\Omega_+}|\nabla w(x,y)|^2\,\dd(x,y)\le 0,
\end{equation}
so the operator $A$ is dissipative. Moreover, proceeding as in the proof of \cite[Theorem~1]{ZhaZua07} we see that $A$ is invertible. Since the resolvent set is open we may find $\lambda>0$ such that $\lambda\not\in\sigma(A)$, so that $\lambda-A$ is invertible and in particular surjective. Hence  $A$ generates a $C_0$-semigroup $\T$ of contractions on $Z$ by the Lumer-Phillips theorem, and it follows from standard semigroup theory that $\sigma(A)\subseteq\{\lambda\in\CC:\R\lambda\le0\}$. In order to see that  $\sigma(A)$ contains no purely imaginary points let $s\in \RR$. We wish to show that for every $z_0=(f,g,h)\in Z$ there exists $(u,v,w)\in D(A)$ such that $(is-A)(u,v,w)=z_0$ and $\|(u,v,w)\|\le C\|z_0\|$ for some $C>0$ which is independent of $z_0$. Let $e_k(y)=\sqrt{2}\sin(k\pi y)$ for $k\in\NN$ and $y\in(0,1)$, recalling that $\{e_k:k\in\NN\}$ is an orthonormal basis for $L^2(0,1)$. We may expand $u$ into a series of the form
$$u(x,y)=\sum_{k=1}^\infty u_k(x)e_k(y),\quad (x,y)\in \Omega_-,$$
with convergence in the norm of $L^2(\Omega_-)$, and we may similarly decompose $v,w, f, g$ and $h$. This gives rise to functions $u_k\in H^2(-1,0)\cap H_{-}^1(-1,0)$, $v_k\in H_{-}^1(-1,0)$, $w_k\in H^2(0,1)\cap H_{+}^1(0,1)$, $f_k\in H^1_-(-1,0)$, $g_k\in L^2(-1,0)$ and $h_k\in L^2(0,1)$, $k\in\NN$, where $H_{-}^1(-1,0)=\{u_0\in H^1(-1,0):u_0(-1)=0\}$ and $H_{+}^1(0,1)=\{w_0\in H^1(0,1):w_0(1)=0\}$. Using orthonormality of the set $\{e_k:k\in\NN\}$, our problem turns into the system of one-dimensional equations 
\begin{equation}\label{eq:res_sys}
\left\{\begin{aligned}
v_k(x)&=is\,  u_k(x)-f_k(x),\quad &x\in (-1,0),\\
u_k''(x)&=(k^2\pi^2-s^2)u_k(x)-is\,f_k(x)-g_k(s),&x\in(-1,0),\\
w_k''(x)&=(k^2\pi^2+is)w_k(x)-h_k(x),&x\in(0,1),
\end{aligned}\right.
\end{equation}
with the boundary conditions $u_k(-1)=w_k(1)=0$, $v_k(0)=w_k(0)$ and $u_k'(0)=w_k'(0)$, $k\in\NN$. These equations can be solved for each $k\in\NN$, and in fact we shall do so explicitly in the proof of Theorem~\ref{thm:ub} below. In order to show that $\|(u,v,w)\|\le C\|z_0\|$ for some $C>0$ we note first that
$$\begin{aligned}
\|z_0\|^2&=\sum_{k=1}^\infty\Big(k^2\pi^2\|f_k\|^2_{L^2}+\|f_k'\|^2_{L^2}+\|g_k\|^2_{L^2}+\|h_k\|^2_{L^2}\Big),
\end{aligned}$$
and we obtain an analogous expression for the norm of $(u,v,w)$. Let us write $z_k=(f_k,g_k,h_k)$ and 
$$\|z_k\|=\big(k^2\pi^2\|f_k\|_{L^2}^2+\|f_k'\|_{L^2}^2+\|g_k\|_{L^2}^2+\|h_k\|_{L^2}^2\big)^{1/2},\quad k\in\NN.$$ It suffices to show that 
\begin{equation}\label{eq:list}
 k\|u_k\|_{L^2},\|u_k'\|_{L^2},\|v_k\|_{L^2},\|w_k\|_{L^2}\lesssim\|z_k\|,\quad k\in\NN,
 \end{equation}
where the implicit constant is independent of $z_0$ and $k$. Since we know that the resolvent set contains a neighbourhood of zero we may assume that $|s|\ge s_0$ for some $s_0>0$. We omit this argument here, since a more careful version of it is presented in the proof of Theorem~\ref{thm:ub} below, where we moreover keep track of how the implicit constant in \eqref{eq:list} depends on $|s|$.
 
It remains to describe the point spectrum $\sigma_p(A)$ of $A$. Let $\lambda\in\CC_-$ and suppose that $z_0=(u,v,w)\in \Ker(\lambda-A)$. Expanding the components of $z_0$ as  above we obtain the system  
\begin{equation}\label{eq:ev_sys}
\left\{\begin{aligned}
v_k(x)&=\lambda u_k(x),\quad &x\in (-1,0),\\
u_k''(x)&=(k^2\pi^2+\lambda^2)u_k(x),&x\in(-1,0),\\
w_k''(x)&=(k^2\pi^2+\lambda)w_k(x),&x\in(0,1),
\end{aligned}\right.
\end{equation}
together with the boundary conditions $u_k(-1)=w_k(1)=0$, $v_k(0)=w_k(0)$ and $u_k'(0)=w_k'(0)$, $k\in\NN$. Since $\lambda\in\CC_-$ we have $p_k(\lambda)\ne0$ for all $k\in\NN$. Let us assume for the moment that we also have $q_k(\lambda)\ne0$, $k\in\NN$. It is then straightforward to show that the ordinary differential equations in \eqref{eq:ev_sys} together with the boundary conditions at $x=\pm1$ imply 
$$\begin{aligned}
u_k(x)&=a_k(\lambda)\sinh(p_k(\lambda) (1+x)),\quad &x\in(-1,0),\\
w_k(x)&=b_k(\lambda)\sinh(q_k(\lambda)(1-x)),&x\in(0,1),
\end{aligned}$$
for some constants $a_k(\lambda),b_k(\lambda)\in\CC$, $k\in\NN$, and now the coupling conditions at $x=0$ can be formulated as
\begin{equation}\label{eq:matrix}
\begin{pmatrix}
\lambda\sinh p_k(\lambda)&-\sinh q_k(\lambda)\\
p_k(\lambda)\cosh p_k(\lambda) &q_k(\lambda)\cosh q_k(\lambda)
\end{pmatrix}
\begin{pmatrix}
a_k(\lambda)\\b_k(\lambda)
\end{pmatrix}=
\begin{pmatrix}
0\\0
\end{pmatrix},\quad k\in\NN.
\end{equation}
Writing $M_k(\lambda)$ for the $2\times2$ matrix appearing in this equation we see that $z_0\ne0$ if and only if $\det M_k(\lambda)=0$ for some $k\in\NN$. On the other hand, if $q_k(\lambda)=0$ for some $k\in\NN$ then $w_k(x)=b_k(\lambda)(x-1)$, $x\in(0,1)$, for some constant $b_k(\lambda)\in\CC$,  and it is straightforward to verify that the coupling conditions at $x=0$ imply $a_k(\lambda)=b_k(\lambda)=0$. The result now follows.
\end{proof}

\section{Resolvent estimates}\label{sec:res}

We now study the growth behaviour of the resolvent norms $\|R(is,A)\|$ as $|s|\to\infty.$ First, in Section~\ref{sec:ub} we obtain an upper bound for the resolvent norms, which we then prove to be optimal in Section~\ref{sec:opt}. These results will eventually allow us to obtain a sharp estimate on the rate of energy decay for classical solutions, which is done  in Section~\ref{sec:energy} below.

\subsection{An upper bound}\label{sec:ub}

We establish the following result.

\begin{thm}\label{thm:ub}
We have $\|R(is,A)\|=O(|s|^3)$ as $|s|\to\infty$.
\end{thm}

\begin{proof}
We use the notation introduced in Section~\ref{sec:sg}. Fix $z_0\in Z$ and $s\in \RR$ with $|s|\ge s_0$, where $s_0>0$ is to be chosen in due course. Let us write $z_0=(f,g,h)$ and $R(is,A)z_0=(u,v,w)$, so that $z_0=(is-A)(u,v,w)$. We may decompose each of the entries of $z_0$ and $R(is,A)z_0$ as in the proof of Theorem~\ref{thm:sg} to obtain the system \eqref{eq:res_sys}
with the boundary conditions listed there. Using the same notation as in the proof of Theorem~\ref{thm:sg} but dropping the subscript $L^2$ from now on, we have $\|v_k\|\lesssim|s|\|u_k\|+\|z_k\|$, so our objective is to show that
\begin{equation*}\label{eq:objective}
 k\|u_k\|,|s|\|u_k\|,\|u_k'\|,\|w_k\|\lesssim|s|^3\|z_k\|,\quad k\in\NN,
\end{equation*}
where the implicit constant is independent of $z_0$, $s$ and $k$. Our approach depends crucially on the relationship between the parameters $s$ and $k$, and we distinguish between two main cases.   In what follows we write $p_k(s)$ for $p_k(\lambda)$ when $\lambda=is$, and similarly for $q_k$.\\[-5pt]

\noindent \emph{Case 1:} $s^2\le k^2\pi^2+1$. We begin with some preliminary manoeuvres. Taking the inner product in $L^2(-1,0)$ of the second equation of \eqref{eq:res_sys} with $v_k$ and integrating by parts we obtain
\begin{equation}\label{eq:uv}
\langle u_k',v_k'\rangle-u_k'(0)\overline{v_k(0)}=\langle is\, f_k+g_k-p_k(s)^2 u_k,v_k\rangle.
\end{equation}
Similarly, taking the inner product in $L^2(0,1)$ of the third equation of \eqref{eq:res_sys} with $w_k$ and integrating by parts we obtain 
\begin{equation}\label{eq:ww}
\|w_k'\|^2+w_k'(0)\overline{w_k(0)}=\langle h_k-q_k(s)^2 w_k,w_k\rangle.
\end{equation}
Adding \eqref{eq:uv} and \eqref{eq:ww} we find, after using the boundary conditions, that
\begin{equation}\label{eq:sum}
\langle u_k',v_k'\rangle+\|w_k'\|^2=\langle is\, f_k+g_k-p_k(s)^2 u_k,v_k\rangle+\langle h_k-q_k(s)^2 w_k,w_k\rangle.
\end{equation}
Expanding this equation using the identity $v_k=is\, u_k-f_k$ and then taking real parts we find, after some crude estimates, that
\begin{equation}\label{eq:real}
 \|w_k\|^2\lesssim \|u_k'\|\|z_k\|+\|z_k\|^2.
\end{equation}
Here we have used that $\|u_k\|\le\frac2\pi\|u_k'\|$ by the Poincar\'e inequality and also that $|s|\lesssim k$. Similarly, if we take imaginary parts in \eqref{eq:sum} and make use of~\eqref{eq:real} then after some standard estimates we obtain
\begin{equation}\label{eq:im}
 \|u'_k\|^2+p_k(s)^2\|u_k\|^2\lesssim \|u'_k\|\|z_k\|+\|z_k\|^2.
\end{equation}
Here we have used that $|s|\ge s_0$ for some $s_0>0$. From the Poincar\'e inequality and the fact that $p_k(s)^2\ge-1$ we see that $\|u_k'\|^2\lesssim \|u_k'\|^2+p_k(s)^2\|u_k\|^2$, so by \eqref{eq:im} we may find a constant $C>0$ such that 
$$\|u_k'\|^2- C\|u'_k\|\|z_k\|-C^2\|z_k\|^2\le0.$$
By computing the roots of the polynomial $t\mapsto t^2-C\|z_k\|t-C^2\|z_k\|^2$ it follows easily that $\|u_k'\|\lesssim\|z_k\|$. By \eqref{eq:real} we also have $\|w_k\|\lesssim\|z_k\|$. Next we estimate $k\|u_k\|$. Supposing for the moment that $p_k(s)^2\ge k^2$ we have $k\|u_k\|\lesssim \|z_k\|$ from \eqref{eq:im} and the estimate for $\|u_k'\|$. On the other hand, if $p_k(s)^2<k^2$ then $k\lesssim|s|$ and hence by the Poincar\'e inequality $k\|u_k\|\lesssim |s|\|u_k'\|\lesssim|s|\|z_k\|$. Finally, we have $|s|\|u_k\|\lesssim k\|u_k\|\lesssim|s|\|z_k\|$.\\[-5pt]

\noindent \emph{Case 2:} $s^2>k^2\pi^2+ 1$. Note that subject to the boundary conditions at $x=\pm1$  the solutions $u_k$, $w_k$ of \eqref{eq:res_sys} are given by
\begin{equation}\label{eq:sol}
\begin{aligned}
u_k(x)&=a_k(s)\sinh(p_k(s)(x+1))+U_{k,s}(x),\quad&x\in(-1,0),\\
w_k(x)&=b_k(s)\sinh(q_k(s)(1-x))+W_{k,s}(x), &x\in(0,1),
\end{aligned}
\end{equation}
where $a_k(s),b_k(s)\in\CC$, $k\in\NN$, are constants and 
$$\begin{aligned}
U_{k,s}(x)&=-\frac{1}{p_k(s)}\int_{-1}^x\sinh(p_k(s)(x-r))H_{k,s}(r)\,\dd r,&x\in(-1,0),\\
W_{k,s}(x)&=-\frac{1}{q_k(s)}\int_{x}^1\sinh(q_k(s)(r-x))h_k(r)\,\dd r, &x\in(0,1),
\end{aligned}$$
with $H_{k,s}(x)=is\,f_k(x)+g_k(x)$, $x\in(-1,0)$. If we denote the $2\times 2$ matrix appearing in \eqref{eq:matrix} by $M_k(s)$ when $\lambda=is$, then we may write the coupling conditions at $x=0$ in the form
$$M_k(s)\begin{pmatrix}
a_k(s)\\b_k(s)
\end{pmatrix}=
\begin{pmatrix}
f_k(0)-is\,U_{k,s}(0)+W_{k,s}(0)\\
W_{k,s}'(0)-U_{k,s}'(0)
\end{pmatrix},\quad k\in\NN.$$
Hence the solutions $u_k$ and $w_k$, $k\in\NN$, may be written, after some elementary but tedious manipulations, in the form
\begin{equation*}\label{eq:u_sol}
\begin{aligned}
u_k(x)&=\frac{is\,q_k(s)\cosh q_k(s)}{p_k(s)\det M_k(s)}\bigg(\frac{p_k(s)}{is}\sinh(p_k(s)(x+1))f_k(0)\\&\qquad\qquad-\sinh(p_k(s)x)\int_{-1}^x\sinh(p_k(s)(1+r))H_{k,s}(r)\,\dd r\\&\qquad\qquad-\sinh(p_k(s)(1+x))\int_x^0\sinh(p_k(s)r)H_{k,s}(r)\,\dd r\bigg)\\
&\qquad+\frac{\sinh q_k(s)}{\det M_k(s)}\bigg(\cosh(p_k(s)x)\int_{-1}^x\sinh(p_k(s)(1+r))H_{k,s}(r)\,\dd r\\&\qquad\qquad +\sinh(p_k(s)(1+x))\int_x^0\cosh(p_k(s)r)H_{k,s}(r)\,\dd r\bigg)\\
&\qquad+\frac{1}{\det M_k(s)}\sinh(p_k(s)(x+1))\int_0^1\sinh(q_k(s)(1-r))h_k(r)\,\dd r
\end{aligned}
\end{equation*}
for $x\in(-1,0)$ and 
\begin{equation*}\label{eq:w_sol}
\begin{aligned}
w_k(x)&=\frac{p_k(s)\cosh p_k(s)}{q_k(s)\det M_k(s)}\bigg(-q_k(s)\sinh(q_k(s)(1-x))f_k(0)\\&\qquad\qquad+\sinh(q_k(s)(1-x))\int_{0}^x\sinh(q_k(s)r)h_{k}(r)\,\dd r\\&\qquad\qquad+\sinh(q_k(s)x)\int_x^1\sinh(q_k(s)(1-r))h_{k}(r)\,\dd r\bigg)\\
&\qquad+\frac{is\sinh p_k(s)}{\det M_k(s)}\bigg(\sinh(q_k(s)(1-x))\int_{0}^x\cosh(q_k(s)r)h_{k}(r)\,\dd r\\&\qquad\qquad+\cosh(q_k(s)x)\int_x^1\sinh(q_k(s)(1-r))h_{k}(r)\,\dd r\bigg)\\
&\qquad+\frac{is}{\det M_k(s)}\sinh(q_k(s)(1-x))\int_{-1}^0\sinh(p_k(s)(1+r))H_{k,s}(r)\,\dd r
\end{aligned}
\end{equation*}
for $x\in(0,1)$. Though somewhat laborious to derive, these formulas can be readily verified simply by substituting them into \eqref{eq:res_sys}. Since $s^2>k^2\pi^2+ 1$, we have $k\lesssim|s|$ and hence $1\le |p_k(s)|\lesssim|s|$ and $|s|^{1/2}\lesssim |q_k(s)|\lesssim|s|$. Moreover, $p_k(s)$ is purely imaginary. Hence if we differentiate the expression for $u_k$ and estimate the $L^2$-norms crudely, and in particular insert a factor of $|s|/k$ for later convenience, we obtain
\begin{equation}\label{eq:ests}
\begin{aligned}
k\|u_k\|,|s|\|u_k\|,\|u_k'\|&\lesssim\frac{s^4 e^{\R q_k(s)}}{k|{p_k(s)}||{\det M_k(s)|}}\|z_k\|,\\\mbox{and}\qquad
\|w_k\|&\lesssim\frac{s^2 e^{\R q_k(s)}}{k|{\det M_k(s)}|}\|z_k\|.
\end{aligned}
\end{equation}
Note that $\R q_k(s)>0$. Next we seek to bound the term $e^{-\R q_k(s)}|\det M_k(s)|$ from below. A straightforward calculation yields
\begin{equation}\label{eq:det_est}
\begin{aligned}
2\frac{|{\det M_k(s)}|}{e^{\R q_k(s)}}&\ge \big|sq_k(s)\sin|p_k(s)|-i|p_k(s)|\cos|p_k(s)|\big|\\
&\quad -e^{-2\R q_k(s)}\big|sq_k(s)\sin|p_k(s)|+i|p_k(s)|\cos|p_k(s)|\big|,
\end{aligned}
\end{equation}
and squaring the  modulus  of first term on the right-hand side we obtain
\begin{equation*}\label{eq:mod}
\begin{aligned}
s^2(\R q_k(s))^2\sin^2|p_k(s)| +\big(s\I q_k(s)\sin|p_k(s)|-|p_k(s)|\cos|p_k(s)|\big)^2.
\end{aligned}
\end{equation*}
If the first summand in this expression is less than  $1/8$ then 
$$|p_k(s)|^2\cos^2|p_k(s)|\ge1-\frac{1}{8s^2(\R q_k(s))^2}\ge\frac12$$ 
because $|s|\R q_k(s)\ge 1/2$, and since $|\I q_k(s)|\le \R q_k(s)$ it follows easily that the second summand in the expression must be at least $1/8$. Hence the first term on the right-hand side of \eqref{eq:det_est} is bounded from below by $1/2\sqrt{2}$. The second term on the right-hand side of \eqref{eq:det_est} converges to zero uniformly in $k$ as $|s|\to\infty$, so if we choose the lower bound $s_0\ge1$ for $|s|$ to be sufficiently large we see that $e^{-\R q_k(s)}|\det M_k(s)|$ is uniformly bounded away from zero. Thus \eqref{eq:ests} becomes
$$k\|u_k\|,|s|\|u_k\|,\|u_k'\|\lesssim\frac{s^4 }{k|p_k(s)|}\|z_k\|\quad\mbox{and}\quad \|w_k\|\lesssim\frac{s^2}{k}\|z_k\|.$$
Now if $s^2\le 2k^2\pi^2$ then $|s|\lesssim k$ and hence
\begin{equation*}\label{eq:s3}
k\|u_k\|,|s|\|u_k\|,\|u_k'\|\lesssim|s|^3\|z_k\|\quad\mbox{and}\quad \|w_k\|\lesssim|s|\|z_k\|.
\end{equation*} 
On the other hand, if $s^2>2k^2\pi^2$ then $|s|\lesssim |p_k(s)|$ and therefore
\begin{equation}\label{eq:subopt}
k\|u_k\|,|s|\|u_k\|,\|u_k'\|\lesssim|s|^3\|z_k\|\quad\mbox{and}\quad \|w_k\|\lesssim s^2\|z_k\|.\\[5pt]
\end{equation}
This completes the proof.
\end{proof}

\begin{rem}\label{rem:ub}\begin{enumerate}[(a)]
\item The explicit solutions to \eqref{eq:res_sys} given in the proof work for all $s\in\RR\setminus\{0\}$. Hence in proving Theorem~\ref{thm:sg}, where we only need to show that the inverse of $is-A$ is a bounded operator for each fixed $s\in \RR\setminus\{0\}$, we may use these explicit solutions to obtain crude bounds for all $k\in\NN$ such that $k^2\pi^2+1<s^2$ and then obtain \eqref{eq:list} by using only Case 1 of the above proof, which crucially does not require the lower bound $s_0$ for $|s|$ to be large.
\item In the above proof we have not attempted to obtain the best possible growth estimate in each of the subcases. For instance, a more careful analysis would show that for $s^2>2k^2\pi^2$ one can improve \eqref{eq:subopt} to 
$$k\|u_k\|,|s|\|u_k\|,\|u_k'\|\lesssim s^2\|z_k\|\quad\mbox{and}\quad \|w_k\|\lesssim |s|\|z_k\|.$$
As we shall see now, however, the overall growth rate of $|s|^3$ is sharp. This suggests that the largest contributions come from terms for which $k\asymp|s|$; see Remark~\ref{rem:opt} below.
\end{enumerate}
\end{rem}

\subsection{Optimality of the upper bound}\label{sec:opt}
We now show that the upper bound obtained in Theorem~\ref{thm:ub} is sharp, by proving that $A$ has eigenvalues which approach the imaginary axis at a suitable rate at infinity.

\begin{thm}\label{thm:opt}
There exist sequences $(\lambda_k^\pm)$ of eigenvalues of $A$ and a constant $C>0$  such that $\I \lambda_k^\pm\sim \pm k\pi$ as $k\to\infty$ and
\begin{equation}\label{eq:sequ}
-\frac{C}{|{\I\lambda_k^\pm}|^{3}}\le\R\lambda_k^\pm<0,\quad k\in\NN.
\end{equation}
In particular,  
 $$\limsup_{|s|\to\infty}\frac{\|R(is,A)\|}{|s|^{3}}>0.$$
\end{thm}

\begin{proof}
Consider the functions $F_k(\lambda)=\tanh p_k(\lambda)$ and 
$$G_k(\lambda)=\frac{p_k(\lambda)}{\lambda q_k(\lambda)}\tanh q_k(\lambda),\quad \lambda\in\CC,$$
where $p_k$, $q_k$ are as in Theorem~\ref{thm:sg}, and let $\smash{\mu_{k}^\pm}=\pm i\pi\sqrt{k^2+1}$, $k\in\NN$. Note that $p_k(\smash{\mu_{k}^\pm})=i\pi$ and let $q_k^\pm=q_k(\smash{\mu_{k}^\pm})$, $k\in\NN$. Then $F_k(\smash{\mu_{k}^\pm})=0$ and 
$$G_k({\mu_{k}^\pm})=\pm\frac{\tanh q_k^\pm}{q_k^\pm\sqrt{k^2+1}},\quad k\in\NN.$$
In particular, since $|q_k^\pm|\sim k\pi$  and $|{\tanh q_k^\pm}|\to1$ as $k\to\infty$  we have $|G_k(\smash{\mu_{k}^\pm})|\sim(k^2\pi)^{-1}$ as $k\to\infty$. In fact, if we let  $\Omega_k^\pm=\{\lambda\in\CC:|\lambda-\smash{\mu_{k}^\pm}|<k^{-3}\}$,  $k\in\NN$, a more careful argument using Taylor expansions shows that  
\begin{equation}\label{eq:G_est}
\sup\big\{|G_k(\lambda)|:\lambda\in\partial\Omega_k^\pm\big\}\sim\frac{1}{k^2\pi},\quad k\to\infty.
\end{equation}
For $\lambda=\smash{\mu_{k}^\pm}+z$  we have $$F_k(\lambda)=i\tan\left(\pi\left(1-\frac{z^2\pm2zi\pi\sqrt{k^2+1}}{\pi^2}\right)^{1/2}\right)$$
and hence another Taylor expansion shows that $|F_k(\lambda)|\ge (2k^2)^{-1}$ for $\lambda\in\partial\Omega_k^\pm$
provided $k$ is sufficiently large. 
Thus by \eqref{eq:G_est} we have $|G_k(\lambda)|<|F_k(\lambda)|$ for such values of $\lambda$ and $k$, so  Rouch\'e's theorem implies that the function $F_k+G_k$ has roots $\lambda_k^\pm\in\Omega_k^\pm$ when $k$ is sufficiently large.  By Theorem~\ref{thm:sg}  any root of $F_k+G_k$ is an eigenvalue of $A$, so we obtain \eqref{eq:sequ}. The final claim follows easily since  $\|R(is,A)\|\ge\dist(is,\sigma(A))^{-1}$, $s\in\RR$.
\end{proof}

\begin{rem}\label{rem:opt}
An alternative approach to proving optimality of the resolvent bound in Theorem~\ref{thm:ub} is to show directly, by estimating $\|R(is,A)z_0\|$ from below for suitable $s\in\RR$ and $z_0\in Z$, that the upper bounds in the latter part of the proof of that result cannot be improved. Note also that the points on the imaginary axis at which the resolvent norm is shown to be large are of the form $\pm is_k$, $k\in\NN$, where $s_k\sim k\pi $ as $k\to\infty$; see Remark~\ref{rem:ub}.
\end{rem}

\section{Energy decay}\label{sec:energy}

We now turn to the rate of energy decay for classical solutions of the abstract Cauchy problem \eqref{eq:ACP} corresponding to the wave-heat system~\eqref{eq:PDE}. Our main result, Theorem~\ref{thm:energy} below, is a consequence of the following abstract result on rates of decay for semigroups on Hilbert space due to Borichev and Tomilov \cite{BT10}; see also \cite{BCT16, RSS17}.

\begin{thm}\label{thm:BT}
Let $Z$ be a Hilbert space and let $\T$ be a bounded $C_0$-semigroup on $Z$ with generator $A$. Suppose that $\sigma(A)\subseteq\CC_-$. Then for any constant $\alpha>0$ the following conditions are equivalent:
\begin{itemize}
    \setlength{\itemsep}{.8ex}
  \item[\textup{(i)}] $\|R(is,A)\|=O(|s|^\alpha)$ as $|s|\to\infty$;
  \item[\textup{(ii)}] $\|T(t)A^{-1}\|=O(t^{-1/\alpha})$ as $t\to\infty$;
  \item[\textup{(iii)}] $\|T(t)z_0\|=o(t^{-1/\alpha})$ as $t\to\infty$ for all $z_0\in D(A)$.
\end{itemize}
\end{thm}

We now come to the main result of this paper.

\begin{thm}\label{thm:energy}
For $z_0\in D(A)$ the energy of the classical solution of \eqref{eq:ACP} satisfies $E_{z_0}(t)=o(t^{-2/3})$ as $t\to\infty$.
\end{thm}

\begin{proof}
Note that for any $z_0\in Z$ we have $E_{z_0}(t)=\frac12\|T(t)z_0\|^2$, $t\ge0$. Since $\|R(is,A)\|=O(|s|^3)$ as $|s|\to\infty$ by Theorem~\ref{thm:ub} it follows from Theorem~\ref{thm:BT} that $E_{z_0}(t)=o(t^{-2/3})$ as $t\to\infty$ for $z_0\in D(A)$.
\end{proof}

\begin{rem}

\begin{enumerate}[(a)]
\item The rate $t^{-2/3}$ in Theorem~\ref{thm:energy} is optimal in the sense that, given any positive function $r$ satisfying $r(t)=o(t^{-2/3})$ as $t\to\infty$, there exists $z_0\in D(A)$ such that $E_{z_0}(t)\ne o(r(t))$ as $t\to\infty$. This follows from Theorem~\ref{thm:opt} and  the uniform boundedness principle together with  \cite[Proposition~1.3]{BD08}; see also \cite[Theorem~4.4.14]{ABHN11}.
\item 
It follows from Theorem~\ref{thm:opt} and  standard $C_0$-semigroup theory that there is no hope of finding a rate of energy decay which is valid for \emph{all} initial values $z_0\in Z$; see also \cite[Lemma~3.1.7]{vN96}.
On the other hand, if $z_0\in D(A^k)$ for some $k\in\NN$ then it follows easily from the semigroup property that $E_{z_0}(t)=o(t^{-2k/3})$ as $t\to\infty$, 
so more regular solutions have faster energy decay.  
\end{enumerate}
\end{rem}

We conclude by mentioning that  methods similar to the ones presented here can be used to obtain sharp estimates for the rate of energy decay in various related problems, such as system \eqref{eq:PDE} but with the coupling condition $u_t(0,y,t)=w(0,y,t)$ replaced by $u(0,y,t)=w(0,y,t)$ for $y\in(0,1)$, $t>0$. Another example is the following wave equation which is damped on one half of its  rectangular domain but not on the other:
\begin{equation*}\label{eq:DWE}
\left\{\begin{aligned}
u_{tt}(x,y,t)&+\mathbbm{1}_{\Omega_+}(x,y)u_t(x,y,t)=\Delta u(x,y,t), &(x,y)\in \Omega,\;t>0,\\[3pt]
u(x,y,t)&=0, &(x,y)\in\partial\Omega,\; t>0,\\[3pt]
u(x,y,0)&=u_0(x,y),\;u_t(x,y,0)=v_0(x,y)&(x,y)\in\Omega.\\
\end{aligned}\right.
\end{equation*} 
Here $\Omega=(-1,1)\times(0,1)$ and $\Omega_+=(0,1)\times(0,1)$ as in Section~\ref{sec:intro}, and  $u_0,v_0$ are suitable functions defined on $\Omega$. In this case the energy 
$$E_{z_0}(t)=\frac12\int_\Omega\Big(|\nabla u(x,y,t)|^2+|u_t(x,y,t)|^2\Big)\,\dd(x,y),\quad t\ge0,$$ 
of any classical solution $u$, with  corresponding initial data $z_0=(u_0,v_0)$, can be shown to satisfy $\smash{E_{z_0}(t)=o(t^{-4/3})}$ as $t\to\infty$, and furthermore this estimate is sharp; see \cite[Part~IV.B]{AnLe14}, \cite{Sta17} and also \cite{BurHit07,	LiuRao05}. Our methods can also be adapted to study the following  wave equation on the square $\Omega_-=(-1,0)\times(0,1)$ subject to Dirichlet boundary conditions along three of its edges but with the coupled heat equation in~\eqref{eq:PDE} replaced by a dissipative boundary condition along the fourth edge:
\begin{equation}\label{eq:WE}
\left\{\begin{aligned}
u_{tt}(x,y,t)&=\Delta u(x,y,t), \hspace{40pt}&(x,y)\in \Omega_-,\;t>0,\\[3pt]
u(x,0,t)&=u(x,1,t)=0, &x\in(-1,0),\; t>0,\\
u(-1,y,t)&=0, &y\in(0,1),\;t>0,\\
u_x(0,y,t)&=-\kappa u_t(0,y,t),& y\in(0,1),\;t>0,\\[3pt]
u(x,y,0)&=u_0(x,y),&(x,y)\in\Omega_-,\\
u_t(x,y,0)&=v_0(x,y),&(x,y)\in\Omega_-,
\end{aligned}\right.
\end{equation}
for suitable initial data $u_0,v_0$ defined on $\Omega_-$ and any constant $\kappa>0$. Models of this type are considered for instance in \cite{AbNi15a, AbNi15b, Lag83}. By formulating \eqref{eq:WE} as an abstract Cauchy problem and proceeding as in Sections~\ref{sec:sg} and \ref{sec:res} one obtains the sharp estimate $\|R(is,A)\|=O(|s|^2)$, $|s|\to\infty$, for the resolvent of the generator $A$ of the corresponding contraction semigroup. It follows as in Theorem~\ref{thm:energy} that the energy of any classical solution $u$ of \eqref{eq:WE} decays like $o(t^{-1})$ as $t\to\infty$, and this estimate too is optimal. 

\bibliographystyle{plain}

\end{document}